\documentclass[11pt,a4paper]{amsart}
\usepackage{amssymb,amsmath,amsthm,eucal}
\usepackage{amscd}
\usepackage{hyperref}
\usepackage[latin1]{inputenc}
\usepackage[all]{xy}
\usepackage{color}

\usepackage[latin1]{inputenc}

\usepackage{mathpazo}
\usepackage[scaled=.95]{helvet}
\usepackage{courier}

\usepackage{color}
\usepackage{hyperref}

\numberwithin{equation}{section}
\theoremstyle{plain} % style de mise en pages plain: normal (thâÂ©orâÂ®me)
\newtheorem{proposition}{Proposition}[section]  % Proposition numâÂ©rotâÂ©e par section
\newtheorem{lemma}[proposition]{Lemma}
\newtheorem{corollary}[proposition]{Corollary} % idem
\newtheorem{theorem}[proposition]{Theorem} % idem

\theoremstyle{definition} % style de mise en page dâÂ©finition
\newtheorem{definition}[proposition]{Definition}
\newtheorem{definitions}[proposition]{Definitions}

\newtheorem{remark}[proposition]{Remark} % Remarque 

\newtheorem{example}[proposition]{Example} % Exemple numâÂ©rotâÂ©s par section
\newtheorem{recalls}[proposition]{Recalls}

\newtheorem{remarks}[proposition]{Remarks}

{}
\newcommand{\ag}{\mathfrak{a}}{}
\newcommand{\bg}{\mathfrak{b}}{}
\newcommand{\cg}{\mathfrak{c}}{}
{}
{}

\def\emptyset{\varnothing}

\newcommand\Edp{\operatorname{E-dp}}
\newcommand\Tcodp{\operatorname{T-codp}}

\newcommand\Tor{\operatorname{Tor}}
\newcommand\Hom{\operatorname{Hom}}

\newcommand\Ext{\operatorname{Ext}}

\newcommand\Ker{\operatorname{Ker}}

\newcommand\im{\operatorname{Im}}

\newcommand{\xx}{\underline x}

\newcommand{\yy}{\underline y}
\newcommand{\zz}{\underline z}

\newcommand{\gam}{\Gamma_{\mathfrak{a}}}

\newcommand{\Cech}{{\Check {C}}_{\xx}}

\newcommand{\qism}{\stackrel{\sim}{\longrightarrow}}

{}

\author[Schenzel, Simon]{Peter Schenzel, Anne-Marie Simon}

\title[torsion of injective modules]{Torsion of injective modules and weakly pro-regular sequences}

\address{Martin-Luther-Universit\"at Halle-Wittenberg,
Institut f\"ur Informatik, D --- 06 099 Halle (Saale), Germany}
\email{schenzel@informatik.uni-halle.de}

\address{Service de G\'{e}om\'{e}trie diff\'{e}rentielle C.P. 218, Universit\'{e} Libre de Bruxelles, Campus Plaine, 
	Boulevard du Triomphe, B - 1050 Bruxelles, Belgium}

\email{amsimon@ulb.ac.be}


\begin{document}
\begin{abstract}
	Let $R$ a commutative ring, $\ag \subset R$ an ideal,  $I$  an injective 
	$R$-module and  $S \subset R$  a multiplicatively closed
	set. When $R$ is Noetherian it is well-known that the $\ag$-torsion sub-module $\Gamma_{\ag}(I)$, 
	the factor module $I/\Gamma_{\ag}(I)$ and the localization $I_S$ are again injective 
	$R$-modules. We investigate these properties in the case of a commutative 
	ring $R$ by means of a notion of relatively-$\ag$-injective $R$-modules. 
	In particular we get another characterization of weakly pro-regular sequences in terms of relatively 
	injective modules. Also we present examples of non-Noetherian commutative rings $R$ and 
	injective $R$-modules for  which the previous properties do not hold. Moreover, under some weak pro-regularity 
	conditions we obtain results of Mayer-Vietoris type. 
\end{abstract}

\date{March $10^{\text{th}}$, 2020}	
\maketitle

\section*{Introduction}

In this note $\ag$ denotes an ideal of a commutative ring $R$. The torsion functor with respect to $\ag$ is denoted by 
$\Gamma_{\ag}(\cdot)$, that is $\Gamma_{\ag}(M)=\{m\in M \mid \ag^tm=0 \text{ for some } t>0\} $. We say that the $R$-module $M$ is $\ag$-torsion if $M=\Gamma_{\ag}(M)$. The right-derived functors of 
$\Gamma_{\ag}(\cdot)$, the local cohomology functors, are denoted by 
$H^i_{\ag}(\cdot)$. When the ring is Noetherian it is well-known that for every 
injective $R$-module $I$, its $\ag$-torsion sub-module $\Gamma_{\ag}(I)$ is again 
injective. Moreover, $I/\Gamma_{\ag}(I)$ and $I_S$ are again injective 
$R$-modules and the natural map $I\to I_S$ is surjective. Here $S\subset R$ denotes a multiplicatively closed set.
 This is no longer the case in general. The main aim of the present paper is to analyse these conditions in the more general case of a commutative ring $R$.

When the ring $R$ is Noetherian it is known that the pair 
$(R, \ag)$ has the following property (referred to as property $B$): for all 
$\ag$-torsion $R$-module $M$ and all $i>0$ we have that $H^i_{\ag}(M)=0$. 
But the case of a Noetherian ring is 
not the only one with this property. For an ideal $\ag$ of a commutative ring $R$ 
we  show in \ref{bisrelinj1} that the pair $(R, \ag)$ has the 
property $B$ if and only if for all relatively-$\ag$-injective $R$-module $J$ it 
holds that $\Gamma_{\ag}(J)$ is again relatively-$\ag$-injective. Here an $R$-module $M$ is called 
relatively-$\ag$-injective if $\Ext^i_R(R/\bg, M)=0$ for all ideal $\bg$ 
containing a power of $\ag$ and all $i\geq 1$.  We also provide an  example of an injective $R$-module $I$ 
such that $\gam (I)$ is not 
relatively-$\ag$-injective (hence also not injective).
The more restricted question to know when $\gam (I)$ is injective has been 
investigated by Quy and Rohrer in \cite{QR}.

When the ideal $\ag$ is finitely generated, say by a sequence 
$\xx=x_1, \ldots, x_k$, we also show in \ref{weak} that the pair 
$(R, \ag)$ has the property $B$ if and only if the sequence $\xx$ is weakly 
pro-regular (see \cite{lip2}, \cite{S} and \cite{SpSa} for more details on the 
notion of weakly pro-regular sequences). That is, we prove another 
characterization of weakly pro-regular sequences (see also \ref{corweak}). Note that the notion of weakly pro-regular sequences plays also an essential r\^ole in  the study of both
completion and local cohomology (see \cite{SpSa} for more details).

Now let $\ag$ and $\bg$ be two finitely generated ideals of a commutative ring $R$. Under some weak pro-regularity conditions we obtain in \ref{mv-5} a long exact Mayer-Vietoris sequence involving these ideals. Under the same conditions, we then prove in \ref{mv-8} that the  the set of ideals $\{\mathfrak{a}^n \cap \mathfrak{b}^n \mid n\geq 1\}$ defines the $(\mathfrak{a} \cap \mathfrak{b})$-adic topology (for a Noetherian ring, this also follows by the Artin-Rees Lemma).

After some investigations about the natural map $M\to M_S$ we provide in \ref{loc-3} another example 
of an injective $R$-module $I$ and a multiplicatively closed set $S \subset R$ such that the localization $I_S$ 
is not injective. That this does not hold in general was first shown  by Dade in \cite{D}. Note that the problems to know when $I_S$ or $\gam (I)$ are injective for an injective module $I$ seem to be related, at least when $S$ consists of the powers of a single element $x$. In that case, we also show in \ref{bisloc-2} that the natural map $M\to M_x$ is surjective if and only if $M/\Gamma_{xR}(M)$ is relatively-$xR$-injective.

In a final section we relate the $\ag$-transform $\mathcal{D}_{\ag}(M)$ of a module 
$M$ to some module of fractions. When the ideal $\ag$ is finitely generated, say by 
a sequence $\xx=x_1, \ldots , x_k$, we provide in \ref{sup-10} and \ref{sup-5} 
natural injections $ \mathcal{D}_{\ag}(M) \hookrightarrow 
H^0(\check{D}_{\xx}^{[1]}\otimes_RM)$  and $H^1_{\ag}(M) \hookrightarrow H^1(\Cech 
\otimes_R M)$, where $\check{D}_{\xx}^{[1]}$ and $\Cech$ denote respectively  the 
global \v{C}ech complex and the \v{C}ech complex. It follows that the deviation of 
the zero-th global \v{C}ech cohomology from the ideal transform is isomorphic 
to the deviation of the first local \v{C}ech cohomology from the first local 
cohomology. We provide an example for which these deviations do not vanish, and we 
note that both deviations vanish when the sequence $\xx$ is weakly  pro-regular.
 
For commutative algebra we refer to \cite{mats}, for homological algebra we also refer to \cite{We}.

\section{Relatively $\ag$-injective modules}

Let $I$ be an injective module over the commutative ring $R$ and let $\ag\subset R$ denote an ideal. To obtain some informations on the submodule $\gam (I)$ a relative notion will be useful.

\begin{definition} \label{relinj0}
	Let $\ag$ be an ideal of a commutative ring $R$. We say that an
	$R$-module $M$ is \textit{relatively-$\ag$-injective} if
	$\Ext^i_R(R/\bg, M)=0$ for all ideal $\bg$ containing a power of $\ag$ and all $i\geq 1$.
\end{definition}

\begin{remarks} \label{rem1} 
	(a) For a relatively-$\ag$-injective $R$-module $M$  we note that $H^i_{\ag}(M)=0$ for all $i\geq 1$. That is because in general we have
	 $H^i_{\ag}(M)=\varinjlim \Ext^i_R(R/\ag^t, M)$. 
	
	(b) Let $0\to M'\to M\to M'' \to 0$ be a short exact sequence of $R$-modules. If both $M'$ and $M$ are 
	relatively-$\ag$-injective, then so is $M''$, as is easily seen by the long exact sequence of the 
	$\Ext^i_R(R/\bg, \cdot)$, where $\bg$ is any ideal containing a power of $\ag$. In particular the short sequences 
	\begin{gather*} 
		0\to \Hom_R(R/\bg,M') \to \Hom_R(R/\bg, M) \to \Hom_R(R/\bg, M'') 
		\to 0 \text { and  }\\ 0\to \gam (M') \to \gam (M)\to \gam (M'') \to 0 
	\end{gather*}  are again exact.
	
	(c) By a standard cohomological argument it follows
	 that the local cohomology of an $R$-module $M$ may be
	computed with a right resolution $J^{\bullet}$ of $M$ consisting of relatively-$\ag$-injective $R$-modules:
	for such a resolution we have $H^i(\gam (J^{\bullet}))\cong H^i_{\ag}(M)$.\\
	(To  see this let  $I^{\bullet}$ be an injective resolution of $M$. Then there is a quasi-isomorphism 
	$f: J^{\bullet}\qism I^{\bullet}$ and the cone $C(f)$ is a left-bounded exact complex of relatively-$\ag$-injective $R$-modules. By view of (b) we have that $\gam (C(f))$ is exact. Hence $\gam (f): \gam (J^{\bullet})\to \gam (I^{\bullet})$ is a quasi-isomorphism. Or see  e.g. \cite{We}  for more details around  cohomological arguments.)
	
	Similarly note also that $\Ext_R^i(R/\bg, M) \cong H^i( \Hom_R(R/\bg, J)$,  where $\bg$ is any ideal containing a power of $\ag$. 
\end{remarks}

For relatively $\ag$-injective modules the Ext-depth and  the Tor-codepth with respect to $\ag$ come also into play.

\begin{definitions} \label{edp} (see also \cite{St90} and \cite{SpSa})
	Let $\ag$ be an ideal of a commutative ring $R$ and $M$ an $R$-module. The 
	Ext-depth  and the Tor-codepth of $M$ with respect to $\ag$ are defined 
	respectively by
	 	\begin{gather*}
		\Edp(\mathfrak a, M)=\inf\{i \mid \Ext^i_R(R/\mathfrak a, M)\neq 0\} \\
		\Tcodp(\mathfrak a, M)=\inf\{i \mid \Tor_i^R(R/\mathfrak a, M)\neq 0\}.
		\end{gather*}  
	where the infimum  is taken over the ordered set $\mathbb N \cup\{\infty\}$. Therefore $\Edp(\ag, M)=\infty$ means that $\Ext_R^i(R/\ag, M)=0$ for all $i \geq 0$. 
	
	Recall that $\Edp(\ag, M)=\inf\{i \in \mathbb{N} \mid H^i_{\ag}(M)\neq 0\}$ and that 
	\begin{gather*}
	\Edp(\ag, M)=\Edp(\ag^n, M) \leq \Edp(\bg, M)\\
	\Tcodp(\ag, M)=\Tcodp(\ag^n, M) \leq \Tcodp(\bg, M) 
	\end{gather*} 
	for all $n\geq 1$ and  any ideal $\bg$ containing a power $\ag^t$ of $\ag$   
	(see \cite[Propositions 5.3.15, 5.3 16, 5.3.11]{St90} and 
	\cite[Chapters 3 and 5]{SpSa}).
	
	When the ideal $\ag$ is finitely generated recall also that 
	$$\Edp(\ag, M)=\infty \text{ if and only if } \Tcodp(\ag, M)=\infty ,$$
	 see \cite[6.1.8]{St90} or \cite[Chapter 5]{SpSa} for a complex version. 
\end{definitions}

In this section we investigate the following question: Given an ideal $\ag$ 
of a commutative ring $R$ and a relatively-$\ag$-injective $R$-module $J$, 
when do we have that $\gam (J)$ also is relatively-$\ag$-injective? We may also wonder when 
the quotient $J/\gam (J)$ is relatively-$\ag$-injective. It turns out that both 
of these properties are equivalent.

\begin{proposition} \label{relinj-2} 
	Let $\ag$ be an ideal of the  commutative ring $R$ and let $J$ denote a relatively-$\ag$-injective $R$-module. The following conditions are equivalent:
	\begin{itemize}
	\item[(i)] $\gam(J)$  is  relatively-$\ag$-injective.
	\item[(ii)] $J/\gam(J)$  is  relatively-$\ag$-injective.
	\item[(iii)]  $\Edp(\ag, J/\gam (J))=\infty$.
	\end{itemize}
\end{proposition}

\begin{proof} 
	We consider the exact sequence 
	$$
	0\to \gam (J) \to J \to J/\gam (J)\to 0
	$$
	and note that the implication (i)$\Rightarrow$(ii) follows by the remark in \ref{rem1} (b). 
	
	For the remaining part of the proof we first note that
	$$
	\Hom_R(R/\ag, \gam (J)) \cong \Hom_R(R/\ag, J) \mbox{ and } 
	\Hom_R(R/\ag, J/\gam (J))=0. 
	$$
	Then we consider the long exact sequence of the $\Ext _R^i(R/\ag, \cdot)$ 
	applied to the above short exact sequence. 
	If $J/\gam(J)$ is relatively-$\ag$-injective, it yields by the previous remark 
	that $\Ext_R^i(R/\mathfrak{a}, J/\Gamma_{\mathfrak{a}}(J)) = 0$ for all $i \geq 0$, i.e. $\Edp(\ag, J/\gam (J))=\infty$ by the definition.
	
	Now assume   $\Edp(\ag, J/\gam (J))=\infty$, so that 
	$\Ext^i_R(R/\bg, J/\gam (J))=0$ for any ideal $\bg$ containing a power of 
	$\ag$ and all $i\geq 0$ (see \ref{edp}). By view of the long exact cohomology sequence of 
	the $\Ext^i_R(R/\bg, \cdot)$ we get the isomorphisms $\Ext_R^i(R/\bg, \gam (J)) \cong \Ext_R^i(R/\bg, J)$ 
	for all $i\geq 0$. Whence $\Ext_R^i(R/\bg, \gam (J))=0$ for all $i>0$ since $J$ is relatively-$\ag$-injective.
	 This finishes the proof.
\end{proof}

\begin{corollary} \label{correlinj-2} 
	Let $\ag$ be an ideal of the  commutative ring $R$ and let $J$ be a relatively-$\ag$-injective $R$-module. Assume that the ideal $\ag$ is finitely generated, say by a sequence $\xx=x_1, \ldots, x_k$.
	If $\gam(J)$  is  relatively-$\ag$-injective, then 
	$$
	J=  (x_1^{n_1},\ldots,x_k^{n_k})J +\gam (J)
	$$
	for every $k$-tupel $(n_1,\ldots,n_k) \in (\mathbb{N}_+)^k$.
\end{corollary} 

\begin{proof} 
	Write $\ag^{\underline{n}}$ for the ideal generated by $x_1^{n_1},\ldots,x_k^{n_k}$ for the $k$-tupel $\underline{n} = (n_1,\ldots,n_k)$. Note that 
	$\ag^{\underline{n}}$ contains a power of $\ag$. If $\gam(J)$  is  relatively-$\ag$-injective, then $\Edp(\ag, J/\gam (J)=\infty$ (see \ref{relinj-2}). Whence 
	$$
	\Tcodp(\ag, J/\gam (J))=\infty=\Tcodp(\ag^{\underline{n}}, J/\gam (J))
	$$ 
	(see \ref{edp}). In particular  
	$J/\gam (J))=\ag^{\underline{n}}\cdot (J/\gam(J))$ and the conclusion follows.
\end{proof}

In the case of a singly generated ideal see also the more precise 
\ref{bisloc-2}.

Here is   the main result of this section. It is a refinement of \cite[Proposition 2.7.10]{SpSa}.

\begin{proposition} \label{bisrelinj1} 
	Let $\ag$ denote any ideal of the  commutative ring $R$. The following conditions are equivalent:
	\begin{itemize}
	\item[(i)] $H^i_{\ag}(M)=0$ for all $i>0$ and any $\ag$-torsion 
	$R$-modules $M$.
	\item[(ii)] $\gam(J)$ is relatively-$\ag$-injective for any relatively-$\ag$-injective $R$-module $J$.
	\item[(iii)] $\gam(I)$  is  relatively-$\ag$-injective for any injective $R$-modules $I$.
	\end{itemize}
\end{proposition}

\begin{proof} 
	(i)$\Rightarrow$(ii):  This is in \cite[Proposition 2.7.10]{SpSa}. Let us recall the  
	proof. Let $J$ denote a relatively-$\ag$-injective $R$-module and put 
	$N=J/\gam (J)$. We consider the short exact sequence
	$$
	0\to \gam (J) \to J \to N\to 0
	$$
	and the associated long exact sequence in local cohomology. By the 
	condition in (i) we have that $H^i_{\ag}(\gam (J))=0$ for all $i\geq 1$. 
	Furthermore, by the hypothesis on $J$ we also have $H^i_{\ag}(J)=0$
	for all $i\geq 1$ (see \ref{rem1}). By the local cohomology long exact sequence 
	it follows that $H^i_{\ag}(N)=0$ for all 
	$i\geq 0$. Hence 
	$$
	\Edp(\ag, N)= \inf\{i \in \mathbb{N} \mid \Ext^i_R(R/\ag, N)\neq 0 \} = \infty,
	$$ 
	see \ref{edp}. Therefore we get that  $\Ext_R^i(R/\bg, N)=0$ for all $i\geq 0$ and 
	all ideals $\bg$ containing a power of $\ag$ (see again \ref{edp}). 
	Now the conclusion follows by the long exact sequence of the 
	$\Ext_R^i(R/\bg, \cdot)$ associated to the above short exact sequence.
	
	(ii)$\Rightarrow$(iii): This is obvious.
	
	(iii)$\Rightarrow$(i):  Let $M$ denote an $\ag$-torsion $R$-module and let $I^0$ denotes its injective hull. There is a short exact sequence 
	$$
	0\to M\to \gam(I^0)\to M_1\to 0
	$$
	where $\gam(I^0)$ is relatively-$\ag$-injective by the condition in (iii).
	Note that the module $M_1$ also is $\ag$-torsion. We iterate the process 
	and obtain a right resolution $J^{\bullet}$ of $M$  by 
	means of $\ag$-torsion relatively-$\ag$-injective modules. 
	Then, by view of \ref{rem1} we have that 
	$$
	H^i_{\ag}(M)\cong H^i(\gam(J^{\bullet})) =H^i(J^{\bullet})=0
	$$ 
	for all $i>0$.
\end{proof}

Here is an explicit example of an injective $R$-module $I$ such that $\gam (I)$ 
is not relatively-$\ag$-injective.

\begin{example} \label{bisexrelinj}
	Let $R = \Bbbk[[x]]$ denote the power series ring in one variable over 
	the field $\Bbbk$. Let $E = E_R(\Bbbk)$ denote the injective hull 
	of the residue field. Then define $S = R \ltimes E$, the idealization  
	of $R$ by its $R$-module $E$. That is, $S=R\oplus E$ as an 
	$R$-module with a multiplication on $S$  defined by $(r, r)\cdot (r', e')=(rr', re'+r'e)$ for all $r, r' \in R$ and $e, e'\in E$.	
	By a result of Faith \cite{Fc} we have that the commutative ring $S$ is self-injective. More precisely, 
	there is an isomorphism of $S$-modules $\Hom_R(S, E)\cong S$  
	(see  \cite[Theorem A.4.6]{SpSa}). We consider the ideal $\ag :=(x, 0)S$ 
	of $S$ and note $\Gamma_{\ag}(S) =0\ltimes E$. Then 
	$\Gamma_{\ag}(S)$ is not injective as an $S$-module (see \cite[2.8.8]{SpSa}).
	
	Here we claim that moreover $\Gamma_{\ag}(S)$ is not relatively 
	$\ag$-injective as an $S$-module. First note that $\ag=xR\ltimes E$ because $E$ 
	is $x$-divisible (mulplication by $x$ is surjective on $E$). Hence $S/\ag \cong \Bbbk$. Since 
	$S/0\ltimes E \cong R$ 
	it follows that the short exact sequence 
	$0\to R \stackrel x \to R \to \Bbbk \to 0$ is also a short exact sequence of 
	$S$-modules. As it is not split exact it yields that $\Ext^1_S(\Bbbk, R)\neq 0$. Then consider the 
	short exact  sequence of $S$-modules
	$$0\to 0\ltimes E \to S \to R\to 0$$
	and the associated long exact sequence of the $\Ext^i_S(\Bbbk, \cdot)$.  It
	induces the exact sequence 
	$$
	0=\Ext_S^1(\Bbbk, S) \to \Ext_S^1(\Bbbk, R) \to \Ext^2_S(\Bbbk, 0\ltimes E).
	$$
	Hence $\Ext^2_S(S/\ag, 0\ltimes E)=\Ext^2_S(\Bbbk, 0\ltimes E)\neq 0$ and $\Gamma_{\ag}(S) =0\ltimes E$ 
	is not relatively-$\ag$-injective. 

	In this example note that the ascending sequence of ideals $0:_S(x, 0)^t = 0 \ltimes 0:_E x^t$, $t>0$
	 does not stabilizes. 
\end{example}
 
We end the section with other properties of relatively-$\ag$-injective $R$-modules, though we do not really need these in this paper.

\begin{proposition}\label{prinj}
	Let $\ag$ be an ideal of a commutative ring $R$ and  $J$
	 a relatively-$\ag$-injective $R$-module. Then the following holds:
	 \begin{itemize}
     \item[(a)] $\Hom_R(R/\bg, J)$ is $R/\bg$-injective for all ideals $\bg$ containing a power of $\ag$.
	 \item[(b)] $\Ext^i_R(N, J)=0$ for all $\ag$-torsion module $N$ and all $i\geq 1$.
 	 \end{itemize}
\end{proposition}

\begin{proof} 
	Let $I^{\bullet}$ denote a injective resolution of $J$. 

	The statement in (a) was already in \cite[2.7.2]{SpSa}. Let us recall the proof. We fix an ideal $\bg$ containing some power of $\ag$. 
	For any ideal $\cg$ containing $\bg$ the complex $\Hom_R(R/\cg, I^{\bullet})$
	is  exact in degree $i>0$ by the assumption on $J$. In particular, the complex
	$\Hom_R(R/\bg, I^{\bullet})$ provides an injective
	$R/\bg$-resolution of the $R/\bg$-module $\Hom_R(R/\bg, J)$. But 
	$\Hom_R(R/\cg, I^{\bullet})=\Hom_{R/\bg}(R/\cg, \Hom_R(R/\bg, I^{\bullet}))$
	as is easily seen. It follows that $\Ext^1_{R/\bg}(R/\cg, \Hom_R(R/\bg, J))=0$. Hence $\Hom_R(R/\bg, J)$ is $R/\bg$-injective by Baer's injectivity criterion.
	
	(b) There is a direct system of augmented complexes 
	$$0\to \Hom_R(R/\ag^t, J)\to \Hom_R(R/\ag^t, I^0)\to \Hom_R(R/\ag^t, I^1)\to \cdots.$$
	These complexes are exact by the assumption on $J$. They are even split-exact  by view of (a). By passing to the direct limit it follows that the complex
	$$0 \to \gam(J)\to \gam (I^0) \to \gam (I^1) \to\cdots$$
	also is split-exact and so is the complex 
	$$0 \to \Hom_R(N, \gam(J))\to \Hom_R(N,\gam (I^0)) \to \Hom_R(N,\gam (I^1)) \to\cdots.$$ Because 
	$\Hom_R(N, I^{\bullet}) \cong \Hom_R(N, \gam(I^{\bullet}))$ when $N$ is $\ag$-torsion it follows $\Ext^i_R(N, J)=0$ for all $i\geq 1$.
\end{proof}

For other informations on relatively-$\ag$-injective modules, see  \cite[Chapter 2, Section 7]{SpSa}. For
other examples, see also the interesting paper of Quy and Rohrer  \cite{QR}.

\section{Weakly pro-regular sequences  
and injectivity}

We shall obtain another characterization of weakly pro-regular sequences. 
Let $\xx = x_1,\ldots,x_k$ denote a sequence of elements of a commutative ring $R$. 
For a natural number $n$ let $\xx^n = x_1^n,\ldots,x_k^n$. For natural numbers 
$m \geq n$ there is a natural map of the Koszul homology modules $H_i(\xx^m;R) \to 
H_i(\xx^n;R)$ for all $i \geq 0$. The sequence $\xx$ is called weakly pro-regular 
provided the inverse system $\{H_i(\xx^n;R) \ n\geq 1\}$ is pro-zero for all 
$i \geq 1$, that is, for all $i >0 $ and any $n$ there is a natural number 
$m \geq n$ such that the natural map $H_i(\xx^m;R) \to H_i(\xx^n;R)$ is zero. 
Weakly pro-regular sequences  were introduced in \cite{lip2} and\cite{S} (see also 
\cite[Chapter 7, Section 3]{SpSa}).

\begin{recalls} \label{rec1} 
	Let $\xx=x_1, \ldots, x_k$ denote a sequence of elements in the commutative 
	ring, write $\ag$ for the ideal generated by this sequence and 	
	$\check{C}_{\xx}$ for the \v Cech complex constructed on this sequence. Then 
	for  all $R$-modules (resp. complexes) $M$ there are natural homomorphisms 
	$H^i_{\ag}(M) \to H^i(\check{C}_{\xx}\otimes_R M)$ and it is known that these 
	are isomorphisms when the sequence $\xx$ is weakly pro-regular (see  
	\cite{lip2} or \cite{S} 	or \cite[7.4.1]{SpSa} applied to an injective 
	resolution of $M$). The main reason for this is given by the following result. 

	\vspace*{.2cm}
	\textit{A sequence $\xx=x_1, \ldots, x_k$ of elements in a 
	commutative ring $R$ is weakly pro-regular if and only $H^i(\check{C}_{\xx}\otimes_R I)=0$ for all injective $R$-modules $I$ and all $i>0$} (see \cite[Lemma 7.3.3]{SpSa}).
	\vspace*{.2cm}
	
	Note that a one length sequence $x$ is weakly pro-regular if and only if the ascending sequence of ideals $0:_R x^t$, $t>0$, stabilizes. Note also that any finite sequence of elements in a Noetherian ring is weakly pro-regular. (This may be proved with the Artin-Rees Lemma, see \cite[4.3.3]{St90} or  \cite[A.2.3]{SpSa}.)

	For more informations on the notion of weakly pro-regular sequences, we refer to \cite{S} or \cite[Chapter 7, Section 3]{SpSa}.
\end{recalls}

Let $I$ denote an injective $R$-module. In the preceding section we investigated when $\gam (I)$ is relatively-$\ag$-injective. The more general question to know when $\gam (I)$ is injective has been investigated by Quy and Rohrer (see \cite{QR}). In this interesting paper Quy and Rohrer proved the following proposition.

\begin{proposition} \label{qr} (\rm{see} \cite[Proposition 3.2]{QR})
	Let $\ag$ be an ideal of the commutative ring $R$ generated by the sequence $\xx=x_1, \ldots , x_k$. If for all injective $R$-modules $I$ one has that $\gam(I)$ is again injective, then the sequence $\xx$ is weakly pro-regular,
\end{proposition}

We note that \ref{qr} is a particular case of the following more precise 
theorem, which also provide us with a new characterization of weakly pro-regular sequences.

\begin{theorem} \label{weak} 
	Let $\ag$ be a finitely generated ideal of a commutative ring $R$ generated by the sequence $\xx=x_1, \ldots, x_k$. Then the following are equivalent:
	\begin{itemize}
	\item[(i)] $H^i_{\ag}(M)=0$ for all $i>0$ and every $\ag$-torsion 
	$R$-module $M$.
	\item[(ii)]  $\xx$ is a weakly pro-regular sequence.
	\item[(iii)] $\gam(I)$  is  relatively-$\ag$-injective for every injective $R$-modules $I$.
\end{itemize}
\end{theorem}

\begin{proof}
	(i)$\Rightarrow$(ii): Let $I$ be any injective $R$-modules and consider the 
	short exact sequence 
	$$
	0\to \gam(I)\to I \to I/\gam(I) \to 0.
	$$
	 By view of the condition (i) and by Proposition 
	\ref{bisrelinj1} we have that $\gam(I)$ is relatively-$\ag$-injective. By view of Proposition \ref{relinj-2}
	we then have that $\Edp(\ag, I/\gam(I))=\infty$.
	
	Now let $K^{\bullet}(\xx^t)$ denote the ascending Koszul complex
	constructed on the sequence $\xx^t=x_1^t, \ldots , x_k^t$ (see \cite{St90} 
	or \cite[Chapter 5 Section 2]{SpSa}). By the Ext-depth sensitivity of the Koszul complex we 
	also have that $H^i(K^{\bullet}(\xx^t) \otimes_R I/\gam(I))=0$ for all $i\geq 0$  
	(see \cite[Theorem 6.1.6]{St90} or \cite[Chapter 5]{SpSa}. 
	It follows that the complex $\check{C}_{\xx}\otimes_R I/\gam(I)$ is exact since 
	$H^i(\check{C}_{\xx}\otimes_R I/\gam(I))\cong \varinjlim H^i(K^{\bullet}(\xx^t) 
	\otimes_R I/\gam(I))$. But the complex $\check{C}_{\xx}$ is a complex of flat  $R$-modules and induces a short exact sequence of complexes
	$$
	0\to \check{C}_{\xx}\otimes_R\gam(I)\to \check{C}_{\xx}\otimes_R I \to \check{C}_{\xx}\otimes_R I/\gam(I)\to 0.
	$$
	Because $\check{C}_{\xx}\otimes_R\gam (I)\cong \gam (I)$ as is easily seen we now obtain that $H^i(\check{C}_{\xx}\otimes_R I)=0$ for all $i>0$. Hence   the sequence $\xx$ is weakly pro-regular by view of the recalls in \ref{rec1}

	(ii)$\Rightarrow$(i): This is rather obvious. Let $M$ be any $\ag$-torsion 
	modules. Then $\check{C}_{\xx}\otimes_R M\cong M$ as is easily seen and the 
	conclusion follows since $H^i_{\ag}(M) \cong H^i(\check{C}_{\xx}\otimes_R M)$ 
	 when the sequence $\xx$ is weakly pro-regular (see \ref{rec1}).

	(i)$\Leftrightarrow$(iii): This was shown in \ref{bisrelinj1}.
\end{proof}

\begin{corollary} \label{corweak} 
	Let $\ag$ be an ideal of the  commutative ring $R$ generated by the sequence $\xx=x_1, \ldots , x_k$. The following conditions are equivalent. 
	\begin{itemize}
	\item[(i)] $J/\gam(J)$  is  relatively-$\ag$-injective for every 
	relatively-$\ag$-injective $R$-module $J$.
	\item[(ii)] $\gam(J)$  is  relatively-$\ag$-injective for every relatively-$\ag$-injective $R$-module $J$.
	\item[(iii)] $\xx$ is a weakly pro-regular sequence.
	\end{itemize} 
\end{corollary}

\begin{proof} 
	The equivalence (i)$\Leftrightarrow$(ii) follows by Proposition 
	\ref{relinj-2}. The equivalence (ii)$\Leftrightarrow$(iii) follows 
	by Theorem \ref{weak}  together with Proposition \ref{bisrelinj1}.
\end{proof}

\section{Weakly pro-regular sequences and Mayer-Vietoris sequences}

Local cohomology is a matter of adic topology, Mayer-Vietoris sequence concerns 
local cohomolody with respect to two distinct ideals. Note that the local homology 
with respect to an ideal $\ag$ only depends on the topological equivalence class of 
$\ag$.  We  say that two ideals $\mathfrak a$ and $\mathfrak a'$ are
topologically equivalent if they define the same adic topology, that is if each 
power of $\ag$ contains a power of $\mathfrak a'$ and vice versa. 
More generally let $\ag$ be an ideal of a commutative ring $R$. We say that a 
set of ideals  $\{ \cg_n \mid  n\in \mathbb N \}$ defines the $\ag$-adic topology 
if these $\cg_n$ form a basis of open neighbourhood of the $\ag$-adic topology, 
that is if each power of $\ag$ contains one the $\cg_n$ and if each $\cg_n$ 
contains a power of $\ag$. 
 
For example, let $\mathfrak{a}, \mathfrak{b} \subset R$ be two ideals of 
the commutative ring $R$. Then  the set $\{\mathfrak{a}^n + \mathfrak{b}^n 
\mid n\geq 1\}$ defines the $(\mathfrak{a}+\mathfrak{b})$-adic topology:  
for each integer $n$ we have the following containment relations 
$(\mathfrak{a}+\mathfrak{b})^{2n}	\subseteq \mathfrak{a}^n + \mathfrak{b}^n 
\subseteq (\mathfrak{a}+\mathfrak{b})^n.$
	
Suppose that the ring $R$ is Noetherian and let $\ag, \bg\subset R$ two ideals. 
Then the ideals $\ag \cap \bg$ and $\ag\cdot \bg$ define the same adic-topology 
(that is because they are finitely generated with the same radical). Recall also 
that the set of ideals 
$\{\mathfrak{a}^n \cap \mathfrak{b}^n \mid n\geq 1\}$ defines 
the $\mathfrak{a} \cap \mathfrak{b}$-adic topology, this follows by the Artin-Rees Lemma (see \cite{BS}). 
We shall see that these facts also hold for a commutative ring under some weak pro-regularity conditions. To this end we need Mayer-Vietoris type results.
 
The following obvious lemma will be useful, its r\^ole in Mayer-Vietoris sequences is central.	

\begin{lemma} \label{sup-1} 
	Let $R$ be a commutative ring. Let $M$ denote an $R$-module with $M_1,M_2 \subset M$ two sub-modules. Then  the short exact sequence
	\[
	0 \to  M\to M\oplus M \to M \to 0
	\]
	defined by $m \mapsto (m, m)$ and $(m', m'') \mapsto m'-m''$, $m, m', m'' \in M$, induces short exact sequences  
	\begin{gather*} 
		0\to M_1\cap M_2 \to M_1\oplus M_2 \to M_1 +M_2 \to 0 \text{ and }\\
		0 \to M/(M_1 \cap M_2) \stackrel{i}{\to} M/M_1 \oplus M/M_2 \stackrel{p}{\to} M/(M_1+M_2) \to 0
	\end{gather*}
\end{lemma}

Here are first results for the torsion of injective modules.

\begin{proposition} \label{presup-2}	
	Let $\mathfrak{a}, \mathfrak{b} \subset R$ two ideals of the commutative ring $R$ and let $M$ be any $R$-module. 
	\begin{itemize}
		\item[(a)] We have 
		$\Gamma_{\mathfrak{a}+\mathfrak{b}}(M) = \Gamma_{\mathfrak{a}}(M) \cap \Gamma_{\mathfrak{b}}(M)$ and  a short exact sequence
		$$
		0 \to \Gamma_{\mathfrak{a}+\mathfrak{b}}(M) \to \Gamma_{\mathfrak{a}}(M) 
		\oplus \Gamma_{\mathfrak{b}}(M) \to \Gamma_{\mathfrak{a}}(M) + \Gamma_{\mathfrak{b}}(M) \to 0.
		$$
		\item[(b)] Moreover, if $M=I$ is injective, then 
		$$
		\Gamma_{\mathfrak{a}}(I) + \Gamma_{\mathfrak{b}}(I) \cong 	\varinjlim \Hom_R(R/\mathfrak{a}^n\cap \mathfrak{b}^n,I).
		$$
	\end{itemize}
\end{proposition}

\begin{proof}
	(a) The inclusion $\Gamma_{\mathfrak{a}+\mathfrak{b}}(M) \subseteq \Gamma_{\mathfrak{a}}(M) \cap \Gamma_{\mathfrak{b}}(M)$ is trivial. 	
	Let $m\in\Gamma_{\mathfrak{a}}(M) \cap \Gamma_{\mathfrak{b}}(M)$ and 
	therefore  $\ag^nm=0=\bg^nm$ for some $n\geq 1$. Then we have 
	$(\ag^n+\bg^n)m=0$ and 
	$(\ag+\bg)^{2n}m=0$, that is  $m\in \Gamma_{\mathfrak{a}+\mathfrak{b}}(M)$.  The short exact sequence follows  (see \ref{sup-1} applied to the sub-modules $\Gamma_{\ag}(M)$ and $\Gamma_{\bg}(M)$ of $M$). \\
	(b) By view of \ref{sup-1} there is the short exact sequences 
	\[
	0 \to \Hom_R(R/\mathfrak{a}^n+\mathfrak{b}^n,I) \to 
	\Hom_R(R/\mathfrak{a}^n,I) \oplus \Hom_R(R/\mathfrak{b}^n,I) 
	\to \Hom_R(R/\mathfrak{a}^n\cap \mathfrak{b}^n,I) \to 0.
	\]
	for all $n \geq 1$. With the natural homomorphisms these form a short exact sequence of direct systems. By passing to the direct limit there is a short exact sequence 
	\[
	0 \to \Gamma_{\mathfrak{a}+\mathfrak{b}}(I) \to \Gamma_{\mathfrak{a}}(I) 
	\oplus \Gamma_{\mathfrak{b}}(I)  \to 
	\varinjlim \Hom_R(R/\mathfrak{a}^n\cap \mathfrak{b}^n,I) \to 0.
	\]
	The last claim follows by  comparing this   short exact sequence with the  one in (a).	
\end{proof}

Then recall that there is also a Mayer-Vietoris long exact sequence for \v Cech cohomology.
		
\begin{theorem}\label{mv-4} (see \cite[Theorem 9.4.3]{SpSa} or \cite{Tc})
	Let $\xx=x_1, \ldots , x_k$ and $\yy=y_1, \ldots , y_l$ be two sequences
	in a commutative ring $R$. We form the sequence $\xx,\yy$ and denote by $\zz$ the sequence formed by the $z_{ij}=x_iy_j$ in any order.
	Let $X$ denote an $R$-complex. Then there is a long exact sequence
	\[
		\ldots \to H^i(\check C_{\xx,\yy}\otimes_RX) \to
		H^i(\Cech\otimes_R X)\oplus H^i(\check C_{\yy}\otimes_RX) \\
		\to H^i(\check C_{\zz}\otimes_RX) \to
		H^{i+1}(\check C_{\xx,\yy}\otimes_RX) \to \ldots
	\]
\end{theorem}

\begin{proof} 
	The proof makes use of the change of rings homomorphism 
	$$
	\mathbb Z[X_1, \ldots , X_k, Y_1, \ldots Y_l] \to R:  
	X_i\mapsto x_i, Y_i \mapsto y_i ,
	$$ 
	where $X_1, \ldots , X_k, Y_1, \ldots Y_l$ are indeterminates, and the  Mayer-Vietoris long exact sequence for Noetherian rings (see e.g. \cite[9.4.2]{SpSa}) together with the recalls in \ref{rec1}. For more  details see \cite[9.4.3]{SpSa}. 
\end{proof}	

We now obtain a Mayer-Vietoris  long exact sequence in a rather large generality.

\begin{theorem} \label{mv-5} 
	With the notations  of \ref{mv-4} put $\ag=\xx R$ and $\bg=\yy R$, 
	so that $\mathfrak{a} \cdot \mathfrak{b} =\zz R$. Suppose  that the 
	three sequences $\xx, \yy, \xx \yy$ are weakly pro-regular. 
	\begin{itemize}
	\item[(a)] The sequence $\zz$ is also weakly pro-regular.
	\item[(b)] There is a long exact sequence
	$$
	\ldots \to H^i_{\mathfrak{a}+\mathfrak{b}}(X)
	\to H^i_{\mathfrak{a}}(X)\oplus H^i_{\mathfrak{b}}(X) \to
	H^i_{\mathfrak{a}\cdot \mathfrak{b}}(X) \to  H^{i+1}_{\mathfrak{a}+\mathfrak{b}}(X) \to \ldots
	$$
	for any $R$-complex $X$.
\end{itemize}
\end{theorem}

\begin{proof} 
	The first statement was already in \cite[Corollary 9.4.4]{SpSa}. It follows by the above theorem together with the characterization of weakly pro regular sequences recalled in \ref{rec1}.
 	Then the  second follows by the first together with Theorem \ref{mv-4}. Note with  \ref{rec1} that the \v Cech cohomology with respect to a weakly pro-regular sequence coincides with the local homology with respect to  the ideal generated by this sequence.
\end{proof}

A second Mayer-Vietoris type result concerns the torsion of injective modules. 

\begin{corollary} \label{mv-6} 
	Let $\ag$ and $\bg$ two finitely generated ideals of the commutative ring $R$, 
	generated respectively by the sequences $\xx=x_1, \ldots , x_k$ and $\yy=y_1, 
	\ldots , y_l$. Assume that the three sequences $\xx, \yy, \xx \yy$ are weakly 
	pro-regular. For any injective $R$-module $I$ we then have a short exact 
	sequence
	\[
	0 \to \Gamma_{\mathfrak{a}+\mathfrak{b}}(I) \to \Gamma_{\mathfrak{a}}(I) 
	\oplus \Gamma_{\mathfrak{b}}(I)  \to 
	 \Gamma_{\mathfrak{a}\cdot\mathfrak{b}}(I)\to 0.
	\]
	Moreover $\Gamma_{\mathfrak{a}\cap\mathfrak{b}}(I)=
	\Gamma_{\mathfrak{a}\cdot\mathfrak{b}}(I)=
	\Gamma_{\mathfrak{a}}(I)+\Gamma_{\mathfrak{b}}(I)$.
\end{corollary}
		
\begin{proof} 
	The short exact sequence is a particular case of the long exact sequence in   \ref{mv-5}. This sequence together with the one in \ref{presup-2}(a) 
	implies that  $ \Gamma_{\mathfrak{a}\cdot\mathfrak{b}}(I)=
	\Gamma_{\mathfrak{a}}(I)+\Gamma_{\mathfrak{b}}(I)$.
 	Then we have the following obvious containment relations
 	$$
 	\Gamma_{\mathfrak{a}\cap\mathfrak{b}}(I) \subseteq \Gamma_{\mathfrak{a}\cdot\mathfrak{b}}(I)=\Gamma_{\mathfrak{a}}(I)+\Gamma_{\mathfrak{b}}(I) \subseteq \Gamma_{\mathfrak{a}\cap\mathfrak{b}}(I),
 	$$
 	they finish the proof.
\end{proof}

Now we are prepared for our last purpose in this section. We need a 
technical lemma.
	
\begin{lemma} \label{mv-7} 
	Let $\ag$ be an ideal of a commutative ring $R$ and let $\{\cg_n\}_{n \geq 1}$ 
	denote a descending sequence of ideals such that $\ag^n\subset \cg_n$ for all 
	$n\geq 1$.Suppose that the natural map 
	$$
	\varinjlim \Hom_R(R/\cg_n,I) \to  \varinjlim \Hom_R(R/\ag^
	n,I)
	$$ 
	is an isomorphism for any injective $R$-module $I$. 
	Then the set of ideals $\{\cg_n \}_{n \geq 1}$ defines the $\ag$-adic topology.
\end{lemma}

\begin{proof} 
	The short exact sequences $0\to \cg_n/\ag^n \to R/\ag^n \to R/\cg_n \to  0$ induce a direct system of short exact sequences
	$$
	0 \to \Hom_R(R/\cg_n, I) \to \Hom_R(R/\ag^n, I) \to \Hom_R(\cg_n/\ag^n , I)\to 0
	$$
	for any injective $R$-module $I$. By passing to the limit there is a short exact sequence 
	$$
	0 \to \varinjlim \Hom_R(R/\cg_n, I) \to \varinjlim \Hom_R(R/\ag^n, I) \to \varinjlim \Hom_R(\cg_n/\ag^n , I)\to 0.
	$$
	Because of our assumption it follows that $\varinjlim \Hom_R(\cg_n/\ag^n , I) =0$ for any  injective $R$-module $I$.
	Now let us fix an $n\in \mathbb N _+$ and let $f : \cg_n/\ag^n\hookrightarrow I^0$ denote an injection into some
 	injective $R$-module $I^0$. Note that $f\in \Hom_R(\cg_n/\ag^n , I^0)$. Because of the vanishing $\varinjlim \Hom_R(\cg_n/\ag^n , I^0) =0$ there must be an integer $m \geq n$ such
 	that the image of $f$ in $\Hom_R(\cg_m/\ag^m, I^0)$ is zero. That is,  
 	the composite of the maps
 	\[
 	\cg_m/\ag^m \to \cg_n/\ag^n  \stackrel{f}\hookrightarrow I^0
 	\]
 	is zero. Since $f$ is an injection it follows that the first map is  zero, and that $\cg_m  \subseteq \ag^n$.
\end{proof}

The following emphasizes again the ubiquity of the weak pro-regularity conditions.

\begin{theorem} \label{mv-8} 
	Let $\ag$ and $\bg$ two finitely generated ideals of the commutative ring $R$, generated respectively by the sequences $\xx=x_1, \ldots , x_k$ and $\yy=y_1, \ldots , y_l$. Assume that the three sequences $\xx, \yy, \xx \yy$ are weakly pro-regular.

   Then the set of ideals $\{\mathfrak{a}^n \cap \mathfrak{b}^n \}_{n \geq 1}$ 
   defines the $\mathfrak{a} \cap \mathfrak{b}$-adic topology as well as the 
   $\mathfrak{a} \cdot \mathfrak{b}$-adic topology. In particular the ideals 
   $\ag\cdot \bg$ and $\ag\cap \bg$ define the same adic topology.
\end{theorem}

\begin{proof} 
	Note first that $(\ag\cdot\bg)^n\subset (\ag\cap\bg)^n \subset \ag^n\cap\bg^n$ for all $n \geq 1$.
	 By view of Corollary \ref{mv-6} we get the natural isomorphisms 
	 \[
	\varinjlim \Hom_R(R/(\ag\cdot \bg)^n, I) \cong \varinjlim \Hom_R(R/(\ag\cap \bg)^n, I) \cong 
	  \Gamma_{\mathfrak{a}}(I)+
	 \Gamma_{\mathfrak{b}}(I).
	 \]
	 By Proposition \ref{presup-2} we have the natural isomorphism
	$$  
	\Gamma_{\mathfrak{a}}(I)+\Gamma_{\mathfrak{b}}(I) \cong \varinjlim \Hom_R(R/\mathfrak{a}^n\cap \mathfrak{b}^n,I).
	$$
	Putting this together it yields natural isomorphisms
	$$\varinjlim \Hom_R(R/(\ag\cdot \bg)^n, I) \cong \varinjlim \Hom_R(R/(\ag\cap \bg)^n, I) \cong 
	\varinjlim \Hom_R(R/\mathfrak{a}^n\cap \mathfrak{b}^n,I)$$
	and the conclusion  follows by Lemma \ref{mv-7}.
\end{proof}

\begin{remark}       
	Let us look again at the long exact sequence in \ref{mv-5}. By view of the above \ref{mv-8} we may now replace in it the product ideal $\ag\cdot \bg$ by the intersection $\ag\cap \bg$. Note that 
	$H^i_{\ag \cdot \bg}(X)\cong H^i_{\ag \cap \bg}(X)$ for all $i$ and all complexes $X$ because the ideals $\ag\cdot \bg$ and $\ag \cap \bg$ define the same adic topology. 
\end{remark}

\section{Modules of fractions and injectivity}

Now let $S$ denote a multiplicatively closed subset in the ring $R$. 
For an $R$-module $M$ we denote by $\iota_{S, M}$
the natural map $M\to M_S$ and by $K_S(M)$ its kernel. There is the question to 
know when the localization $I_S$ of an injective $R$-module $I$ is again injective. 
This was claimed to be true by Rotman (see \cite[3.76]{Rj}) and shown to be 
incorrect by Dade (see his interesting paper \cite{D}), though Dade did not provide 
an explicit example (a first concrete example of an injective module that does not 
localize can be found in \cite[A.5.4]{SpSa}).
This problem seems to be related to the question when $\Gamma_{\mathfrak{a}}(I)$ 
is an injective $R$-module.  Note that $K_S(M) = \sum_{s\in S} \Gamma_{sR}(M)$. We now investigate a little bit in this direction and provide a further example.

The following elementary lemma will be used repeatedly.

\begin{lemma} \label{loc-1}
	Let $S \subset  R$ denote a multiplicatively closed set in the commutative ring $R$. For an $R$-module $M$ the natural 
	homomorphism 
	\[
	\iota_{S, M} : M \to M_S, \; m \mapsto m/1,
	\]
	is surjective if and only if $M/K_S(M) =s\cdot M/K_S(M)$  for all $s \in S$. 
	The last condition is equivalent to $M = K_{S}(M) + s\cdot M$ for  all $s \in S$.
\end{lemma}

\begin{proof}
	First note that the last equivalence is easily seen. 
	
	If the map $\iota_{S, M}$ is surjective, then $M_S\cong M/K_S(M)$, in particular $M/K_S(M) =s\cdot M/K_S(M)$  for all $s \in S$. 
	
	Conversely, assume that $M/K_S(M) =s\cdot M/K_S(M)$  for all $s \in S$. This means that the muliplications by any $s\in S$ are surjective on $M/K_S(M)$. As  they are also injective on $M/K_S(M)$ it follows that $M/K_S(M)\cong M_S$. Whence $\iota_{S, M}$ is surjective.
\end{proof}

\begin{proposition} \label{loc-2} 
	Let $S$ denote a multiplicatively closed set in the commutative ring $R$ and let $M$ be an $R$-module. If the $R$-module $M/K_S(M)$ is injective, then $M/K_S(M)\cong M_S$. 
\end{proposition}

\begin{proof}
	For all $s\in S$ the multiplication by $s$ on $M/K_S(M)$ 
	is always injective. When the $R$-module $M/K_S(M)$ is injective it follows that these multiplications
	 are also surjective. Whence $\iota_{S, M}$ is surjective (see \ref{loc-1}) and 
	$M/K_S(M)\cong M_S$.
\end{proof}

\begin{corollary} \label{bisloc-1} 
	Let $S$ denote a multiplicatively closed set in the commutative ring $R$ and let $I$ denote an injective $R$-module. Assume that $K_S(I)$ is injective. 
	Then the localized module $I_S$ is injective and $I_S\cong I/K_S(I)$.
\end{corollary}

\begin{proof} 
	By the assumption on $K_S(I)$ the short exact sequence 
	$$
	0 \to K_S(I) \to I \to I/K_S(I) \to 0
	$$
	is split exact and $I/K_S(I)$ is injective. We conclude by \ref{loc-2}.
\end{proof}

In the case when the multiplicatively closed subset $S$ of $R$ consists of the powers of a single element we have a refinement of Proposition \ref{loc-2}
 
\begin{proposition}\label{bisloc-2}
	Let $x \in R$ be an element. For an $R$-module $M$ the following conditions 
	are equivalent:
	\begin{itemize}
		\item[(i)] The $R$-module $M/\Gamma_{\xx R}(M)$ is relatively-$xR$-injective.
		\item[(ii)]  $\Edp(xR, M/\Gamma_{\xx R}(M))=\infty$.
		\item[(iii)] $M/\Gamma_{\xx R}(M)=x\cdot M/\Gamma_{\xx R}(M)$.
		\item[(iv)]  The natural map $M\to M_x$ is surjective.
	\end{itemize}
\end{proposition}

\begin{proof}
	(i)$\Rightarrow$(ii): This is clear, note that always $\Hom_R(R/xR, M/\Gamma_{\xx R}(M))=0$.
 
	(ii)$\Rightarrow$(iii): With the condition in (ii) we also have $\Tcodp(xR,  M/\Gamma_{\xx R}(M)) =\infty$, 
	see \ref{edp}. This implies condition (iii).

	(iii)$\Rightarrow$(iv): This follows by \ref{loc-1}. 

	(iv)$\Rightarrow$(i): With the condition in (iv) we have $M/\Gamma_{\xx R}(M) \cong M_x$. Then  note that
 	always $\Ext_R^i(R/xR, M_x)=0$ for all $i\geq 0$ because multiplication by $x$ acts as an automorphism on $M_x$.
\end{proof}

We are ready for the discussion of the following 
example. Note first that an  $R_S$-module is $R$-injective if and only if it is 
$R_S$-injective (see \cite{D} or \cite[A.5.1]{SpSa}). That is because $\Hom_R(M,N) 
\cong \Hom_{R_S}(M,N)$ for two $R_S$-modules $M,N$ and because $R_S$ is $R$-flat.

\begin{example} \label{loc-3} 
	(A) Let $\Bbbk$ denote a field and $x,y$ two variables over $\Bbbk$. Let $R = \Bbbk[|x,y|]$ denote the formal power series ring over $\Bbbk$. 
	Let $E = E_R(R/\mathfrak{m})$ denote 
	the injective hull of the residue field.  We define $S = R \ltimes E$ as the 
	idealization of $R$ by $E$. Then $\Hom_R(S,E)$ is an injective $S$-module 
	and $\Hom_R(S,E) \cong S$ as follows by a Theorem of Faith (see \cite{Fc} or  \cite[A.4.6]{SpSa}). 
	
	Let $(x,0) \in S$. We claim that the localization 
	$S_{(x,0)}$ is not an injective $S$-module. This is equivalent to 
	the fact that $S_{(x,0)}$ is not self-injective (see above).  
	But $S_{(x,0)} \cong R_x$ as is easily seen and $R_x$ is not self-injective. \\
	(B) Moreover, let $\check{C}_{(x,0)}: 0 \to S \to S_{(x,0)} \to 0$ 
	denote the \v{C}ech complex of $S$ with respect to the one length sequence $(x,0)$. 
	Then it follows that 
	\[
	\Gamma_{(x,0)S}(S) = H^0(\check{C}_{(x,0)}) = 0 \ltimes E, \; \mbox{ and } \;
	H^1(\check{C}_{(x,0)}) = (R_x/R) \ltimes 0,
	\]
	so that $H^1(\check{C}_{(x,0)}(S)) \not= 0$ for the injective 
	$S$-module $S$.
	
	Finally, because the natural map $S\to S_{(x, 0)}$ is not surjective, the module $S/ \Gamma_{(x, 0)S}$ is 
	not isomorphic to $S_{(x, 0)}$. Whence it follows by \ref{loc-2} that $S/\Gamma_{(x,0)}(S)$ and also 
	$\Gamma_{(x,0)}(S)$ are not $S$-injective modules.
	Moreover  $S/\Gamma_{(x,0)}(S)$ and  
	$\Gamma_{(x,0)}(S)$ are not relatively-$(x,0)S$-injective modules by \ref{bisloc-2} and \ref{relinj-2}. \\
	(C) Another feature of the example in (A) is the following. Let $S^{\mathbb{N}}$ denote the direct product 
	of copies of $S$ over the index set $\mathbb{N}$. Then $S^{\mathbb{N}}$ is  an injective $S$-module 
	(as a direct product of injective 
	$S$-modules). Now let $S^{(\mathbb{N})}$ be the 
	direct sum of copies of $S$ over the index set $\mathbb{N}$. We claim that  
	$S^{(\mathbb{N})}$ is not an injective $S$-module. To this end look at 
	the short exact sequence 
	\[ 
	0 \to S^{(\mathbb{N})} \to S^{(\mathbb{N})} \to S_{(x,0)} \to 0
	\]
	as it is derived from the isomorphism $S_{(x,0)} \cong \varinjlim \{S_n,(x,0) \}$ 
	with the direct system $S_n = S$ and $S_n \to S_{n+1}$ multiplication 
	by $(x,0)$ for all $n \geq 1$. Assuming that $S^{(\mathbb{N})}$ is $S$-injective it implies that the above sequence is split exact and 
	therefore $S_{(x,0)}$ is $S$-injective. This is not true by (A).
\end{example}

There is the more general question to know when the natural homomorphism $I\to I_S$ 
is surjective for an injective module $I$. Note that this is not always the case, 
as shown by example  \ref{loc-3} (B). In the case when the multiplicatively closed 
subset $S$ of $R$ consists of the powers of a single element the answer is rather 
simple.

\begin{proposition} \label{loc-4} 
	Let $x$ denote an element of the commutative ring $R$. The following conditions are equivalent:
	\begin{itemize}  
		\item[(i)] The natural map $J\to J_x$ is surjective  for any relatively-$xR$-injective module $J$.
		\item[(ii)] The natural map $I\to I_x$ is surjective for any injective module $I$.
		\item[(iii)] The one length sequence $x$ is weakly pro-regular.
	\end{itemize}
\end{proposition} 

\begin{proof} 
	(i)$\Rightarrow$(ii): This is obvious.

	(ii)$\Rightarrow$(iii): If $I\to I_x$ is onto, then $H^1(\check{C}_x\otimes_R I)=0$. The claim follows by the recalls in \ref{rec1}.  

	(iii)$\Rightarrow$(i): Assume that the one length sequence $x$ is weakly 
	pro-regular and let $J$ be any relatively-$xR$-injective $R$-module. On one hand we 
	have $H^1(\check{C}_x\otimes_R J)=H^1_{xR}(J)$ (see \ref{rec1}). 
	On the other hand we also have $H^1_{xR}(J)=0$ (see \ref{rem1}(a)). Whence  the 
	natural map $J\to J_x$ is onto.
\end{proof}

\section{Ideal transforms}
In the final section we relate our results to the  ideal transforms.

\begin{definition} \label{atrans}  
	For an ideal $\mathfrak{a}$ of a commutative ring $R$ and an $R$-module $M$ the
	$\mathfrak{a}$-transform of $M$ is defined by  
	$\mathcal{D}_{\mathfrak{a}}(M) = \varinjlim \Hom_R(\mathfrak{a}^n,M)$.
	
	For this and related results we refer to \cite{BS} and \cite[Chapter 12, section 5]{SpSa}. 
\end{definition}

First we shall discuss in more detail the $\mathfrak{a}$-transform of 
an $R$-module $M$.

\begin{proposition} \label{sup-9}
	Let $\mathfrak{a} \subset R$ denote an ideal of the commutative ring $R$. Let $M$ be an $R$-module. 
	\begin{itemize}
		\item[(a)] There is a natural homomorphism $\tau_M: M\to \mathcal{D}_{\ag}(M)$ and a short exact sequence 
		\[
		0 \to \Gamma_{\mathfrak{a}}(M) \to M \stackrel{\tau_M}{\longrightarrow} 
		\mathcal{D}_{\mathfrak{a}}(M) \to H^1_{\mathfrak{a}}(M) \to 0.
		\]
		\item[(b)] Let $S$ denote a multiplicatively closed subset of $R$ 
		such that $\mathfrak{a} \cap S \not= \emptyset$. There is a natural homomorphism 
		$\xi_M :\mathcal{D}_{\mathfrak{a}}(M) \to M_S$ such that $\xi_M \circ \tau_M = \iota_{S,M} : M \to M_S$.	
		\item[(c)] If $M=J$ is relatively-$\mathfrak{a}$-injective, then $\mathcal{D}_{\ag}(J)\cong J/\gam (J)$.
	\end{itemize}
\end{proposition}

\begin{proof}
	The proof of (a) is well-known. We only need to take the direct 
	limit of the direct system of exact sequences with the natural maps
	\[
	0 \to \Hom_R R/\ag ^n, M) \to M \to \Hom_R(\ag ^n, M) \to \Ext_R^1(R/\ag^n, M) \to 0.
	\]
	For the proof of (b) we choose first an element $x \in \mathfrak{a} \cap S$. Then we define homomorphisms
	\[ 
	\Hom_R(\ag^n, M)\to M_S \text{ by } g_n\mapsto g_n(x^n)/x^n .
	\]
	They do not depend on the particular choice of $x$ as easily seen. Moreover they provide a direct system of sequences 
	\[ 
	M\to \Hom_R(\ag^n, M)\to M_S.
	\] 
	We take  direct limits and obtain the wanted homomorphism 
	$\xi_M :\mathcal{D}_{\mathfrak{a}}(M) \to M_S$ such that 
	$\xi_M \circ \tau_M = \iota_{S,M} : M \to M_S$.
	
	The proof of (c) follows by (a) since 
	$H^1_{\ag}(J) \cong \varinjlim \Ext_R^1(R/\ag^n,J) =0$ for a relatively-$\ag$-injective $R$-module $J$.
\end{proof}
	
Assume now that the ideal $\ag$ of $R$ is finitely generated, 
say by the sequence $\xx=x_1, \ldots , x_k$. We define the $R$-complex 
$\check{D}_{\xx}$ as the kernel of the natural surjective map 
$\Cech \to R$   (see \cite[6.1.6]{SpSa} for more details).
For an $R$-module $M$ we have a natural injection  
$M/\gam (M) \hookrightarrow H^1(\check{D}_{\xx}\otimes_R M)$.
This follows because the complex $\Cech \otimes_R M$ has the form 
$$
\Cech \otimes_R M : 
0\to M \stackrel{d^0}{\longrightarrow} \oplus_{i=1}^r M_{x_i} 
	\stackrel{d^1}{\longrightarrow} \oplus_{1 \leq i < j \leq k} M_{x_ix_j}
	\stackrel{d^2}{\longrightarrow} \ldots \to M_{x_1\cdots x_k} \to 0
$$ 
and because $M/\gam (M) \cong \im (d^0) \subseteq \Ker (d^1) = H^1(\check{D}_{\xx}\otimes_R M)$.

\begin{proposition} \label{sup-10}
	Let $\xx = x_1,\ldots,x_k$ denote a sequence of elements and $\ag = \xx R$. 
	For any $R$-module $M$ there is an injective  homomorphism 
	$\rho_M : \mathcal{D}_{\ag}(M) \to \oplus_{i=1}^k M_{x_i}$
	such that $\rho_M \circ \tau_M$ is the natural map $ M \to \oplus_{i =1}^k M_{x_i}$ 
 	Moreover it induces an injection  
 	$$ 
 	\mathcal{D}_{\ag}(M) \hookrightarrow H^1(\check{D}_{\xx}\otimes_RM).
 	$$ 
\end{proposition}

\begin{proof} 
	By view of \ref{sup-9} (b) there are homomorphisms $\xi^i_M : \mathcal{D}_{\ag}(M) \to M_{x_i}$ for $i = 1,\ldots,k$ such that the composite $M\stackrel{\tau_M}\to \mathcal{D}_{\ag}(M) \to M_{x_i}$ is the natural map $M\to M_{x_i}$. Then the homomorphism $\rho_M$ is defined by $\mathcal{D}_{\ag}(M) \to 
	\oplus_{i=1}^k M_{x_i}, f \mapsto (\xi_M^i(f))_{i=1}^k$.
	
	In order to show that $\rho_M$ is injective suppose that 
	$f \in \mathcal{D}_{\ag}(M)$ maps to zero, i.e. $(\xi_M^i(f))_{i=1}^k = 0$. 
	Let $g_n \in \Hom_R(\mathfrak{a}^n,M)$ denote a representative of $f$. 
	Then $g_n(x_i^n)/x_i^n =0$ for $i = 1,\ldots,k$. Therefore, there is an 
	integer $m \geq n$ such that $x_i^mg_n(x_i^{n}) = 0$ for all $i = 1,\ldots,k$.  
	Because of $\ag^{n+m+k}\subseteq (x_1^{n+m},\ldots,x_k^{n+m})R$
	this implies $g_n(\ag^{m+n+k})=0$. Whence the restriction of $g_n$ to  
	$\ag^{m+n+k}$ is zero and therefore $f = 0$.
	
	For the final claim we have that $H^1(\check{D}_{\xx}\otimes_RM) = 
	\Ker (\oplus_i M_{x_i} \stackrel{d^1}\to \oplus_{i < j} M_{x_ix_j})$
    and note that $\mathcal{D}_{\ag}(M) \subseteq \Ker d^1$. Indeed let 
    $f \in {\mathcal{D}}_{\ag}(M)$ and $g_n \in \Hom_R(\ag^n,M)$ a 
    representative of $f$. This $f$ is mapped to 
	$(g_n(x_i^n)/x_i^n)_{i=1}^k \in \oplus_{i =1}^k M_{x_i}$ which is well 
	defined and belongs to $\Ker d^1$ since in $M_{x_ix_j}$ we have the 
	equalities $g_n(x_i^n)/x_i^n = g_n(x_j^n)/x_j^n$ for all $i, j \in \{1,\ldots,k\}$. 
\end{proof}

\begin{corollary} \label{sup-5}	
	Let $\xx$ and $\ag$ be as in \ref{sup-10}. Let $M$ denote an $R$-module.
	\begin{itemize}
		\item[(a)] 
		The natural map $H^1_{\ag}(M) \to H^1(\Cech \otimes_R M)$ is injective and there is a commutative diagram with exact rows
		\[
		\xymatrix{
			0 \ar[r] & \Gamma_{\ag}(M) \ar[r]  \ar@2{-}[d] & M \ar[r] \ar@2{-}[d] & \mathcal{D}_{\ag}(M) 
			\ar[r] \ar@{^{(}->}[d] & H^1_{\ag}(M )\ar[r] \ar@{^{(}->}[d] & 0 \\
			0 \ar[r] & H^0(\check{C}_{\xx}\otimes_RM) \ar[r]  & M \ar[r]  & H^1(\check{D}_{\xx}\otimes_RM)
			\ar[r] & H^1(\check{C}_{\xx}\otimes_RM)\ar[r]  & 0.
		}
		\]
		In particular there is an isomorphism $H^1(\check{D}_{\xx}\otimes_RM)/\mathcal{D}_{\ag}(M) \cong 
		H^1(\check{C}_{\xx}\otimes_RM)/H^1_{\ag}(M )$.
		\item[(b)] The natural injection $M/\gam (M) \hookrightarrow H^1(\check{D}_{\xx}\otimes_R M)$ factors through the injection 
		$\mathcal{D}_{\ag}(M) \hookrightarrow H^1(\check{D}_{\xx}\otimes_RM)$.
		\item[(c)] 
		Suppose that the sequence $\xx$ is weakly pro-regular. 
		Then we have the isomorphism 
		$$ 
		\mathcal{D}_{\ag}(M) \cong H^1(\check{D}_{\xx}\otimes_RM).
		$$
	\end{itemize}
\end{corollary}

\begin{proof} 
	(a)	There is a short exact sequence  $0 \to \check{D}_{\xx}\otimes_RM 
	\to \check{C}_{\xx}\otimes_RM \to M \to 0$ of complexes (see 
	\cite[6.1.6]{SpSa}). The short exact sequence at the bottom of the diagram 
	follows by the associated long exact sequence in cohomology, while the one at 
	the top   is shown in \ref{sup-9} (a). The commutativity of the diagram is 
	rather obvious. Because the third vertical  map $ \mathcal{D}_{\ag}(M) \to 
	H^1(\check{D}_{\xx}\otimes_RM)$ is injective (see \ref{sup-10}), so is the 
	fourth one. The last statement in (a) follows now. 
	
	(b) The statement follows by the diagram.

	(c) If the sequence $\xx$ is weakly pro-regular the fourth map in the diagram is an isomorphism (see \ref{rec1}). Whence the third map also is an isomorphism.
\end{proof}

The isomorphism in \ref{sup-5} (c) originally proved for a Noetherian ring  is 
known as Deligne's formula (see \cite[Exercise 3.7]{Hr} or \cite[20.1.14]{BS}). It 
does not hold necessarily either for non-Noetherian rings nor for injective modules.

\begin{example} 
	Take the ring $S=R\ltimes E$ of Example \ref{bisexrelinj} and take 
	$\ag=(x,0)S$. We noted that $S$ is self-injective and with 
	$\gam (S) = 0 \ltimes E$. Hence $\mathcal{D}_{\ag}(S) =R \subsetneqq
	H^1(\check{D}_{(x, 0)}\otimes _SS) =S_{(x, 0)}=R_x$. 
\end{example}

\textit{Note added in proof.} Theorem \ref{weak} has been shown independently 
by R. Vyas and A. Yekutieli (see: Weak Stability, and the Noncommutative MGM Equivalence, J. Algebra 513 (2018), 265-325), where the notion of "flasque module" 
is used instead of "relative injective". Thanks to A. Yekutieli for drawing our attention to their paper.

\end{document}